\documentclass{amsart}
\usepackage{amsfonts,amssymb,amscd,amsmath,enumerate,verbatim,calc}
\usepackage[all]{xy}

\newcommand{\CM}{Cohen-Macaulay}

\newcommand{\wrt}{with respect to}
\newcommand{\B}{\mathcal{B} }

\newcommand{\n}{\mathfrak{n} }
\newcommand{\m}{\mathfrak{m} }
\newcommand{\M}{\mathfrak{M} }
\newcommand{\q}{\mathfrak{q} }

\newcommand{\R}{\mathcal{R} }
\newcommand{\Sc}{\mathcal{S} }
\newcommand{\Z}{\mathbb{Z} }

\newcommand{\rt}{\rightarrow}
\newcommand{\xar}{\longrightarrow}
\newcommand{\ov}{\overline}

\newcommand{\wt}{\widetilde }

\newcommand{\image}{\operatorname{image}}

\newcommand{\depth}{\operatorname{depth}}

\newcommand{\amp}{\operatorname{amp}}

\newcommand{\coker}{\operatorname{coker}}

\theoremstyle{plain}

\newtheorem{theorem}{Theorem}[section]

\newtheorem{proposition}[theorem]{Proposition}

\theoremstyle{definition}

\newtheorem{s}[theorem]{}
\newtheorem{remark}[theorem]{Remark}
\newtheorem{example}[theorem]{Example}

\theoremstyle{remark}

\begin{document}

\title[Bockstein cohomology ]
{Bockstein cohomology of associated graded rings}
\author{Tony~J.~Puthenpurakal}
\date{\today}
\address{Department of Mathematics, IIT Bombay, Powai, Mumbai 400 076, India}

\email{tputhen@math.iitb.ac.in}

\begin{abstract}
Let $(A,\m)$ be a \CM \ local ring of dimension $d$ and let $I$ be an $\m$-primary ideal. Let $G$ be the associated graded ring of $A$ \wrt \ $I$ and let $\R = A[It,t^{-1}]$   be the extended Rees ring of $A$ \wrt \  $I$.  Notice $t^{-1}$ is a non-zero divisor on $\R$ and $\R/t^{-1}\R = G$. So we have \textit{Bockstein operators} $\beta^i \colon H^i_{G_+}(G)(-1) \rt H^{i+1}_{G_+}(G)$ for $i \geq 0$. Since $\beta^{i+1}(+1)\circ \beta^i = 0$ we have \textit{Bockstein cohomology} modules $BH^i(G)$ for $i = 0,\ldots,d$. In this paper we show that certain natural conditions on  $I$ implies vanishing of some Bockstein cohomology modules.
\end{abstract}

\maketitle
\section*{Introduction}
Let $(A,\m)$ be a \CM \ local ring of dimension $d$ and let $I$ be an $\m$-primary ideal. The \emph{Hilbert function} of $A$ \wrt \ $I$ is 
$H^I(A,n) = \ell(I^n/I^{n+1})$.  Here $\ell(-)$ denotes length as an $A$-module. A fruitful area of research has been to study the interplay between Hilbert functions and properties of blowup algebra's of $A$ \wrt \ $I$, namely the \textit{associated graded ring} $G_I(A) = \bigoplus_{n \geq 0} I^n/I^{n+1}$, the \textit{Rees ring} $\Sc(I) = \bigoplus_{n \geq 0}I^n$ and the \textit{extended Rees ring } $\R(I) = \bigoplus_{n \in \Z}I^n$ (here $I^n = A$ for $n \leq 0$ and $\R(I)$ is considered as a subring of $A[t,t^{-1}]$). See the texts
 \cite[Section 6]{VaSix} and  \cite[Chapter 5]{VasBook} for nice surveys on this subject. Graded local cohomology  has played an important role in this subject.
 For  various applications
   see  \cite[4.4.3]{BH},\cite{Durham}, \cite{VerJoh},
\cite{Blanc},  \cite{ItN}, \cite{Tr}
   and \cite{HMc}. 

Set $G = G_I(A)$. Let $H^i(G)$ denote $i^{th}$-local cohomology module of $G$ \wrt \ $G_+ = \bigoplus_{n>0}I^n/I^{n+1}$. 
Notice $t^{-1}$ is a non-zero divisor on $\R(I)$ and $\R(I)/t^{-1}\R(I) = G$. So we have \textit{Bockstein operators} $\beta^i \colon H^i(G)(-1) \rt H^{i+1}(G)$ for $i \geq 0$. Since $\beta^{i+1}(+1)\circ \beta^i = 0$ we have \textit{Bockstein cohomology} modules $BH^i(G)$ for $i = 0,\ldots,d$. Despite being natural, Bockstein cohomology groups of associated graded rings have not been investigated before. The goal of this paper is to compute it in some cases.

It is well known that for $n \gg 0$ we have $\depth G_{I^n}(A) \geq 1$. It can occur that $\depth G_{I^n}(A) = 1$ for all $n \gg 0$; see \cite[7.13]{Pu5}. Bockstein cohomology fares better. In Theorem \ref{asymp} we prove that if $d \geq 2$ then for
for all $n \gg 0$ we have  
$BH^i(G_{I^n}(A)) = 0$ for $i = 0,1$.

The formal series $H^I(A,z) = \sum_{n\geq 0}H^I(A,n)z^n$ is called the Hilbert series of $A$. It is well-known that
\[
H^I(A,z) = \frac{h_I(z)}{(1-z)^d} \quad \text{where} \ h_I(z) \in \Z[z], \ \text{and} \  h_I(1) > 0.
\]
If $f$ is a polynomial  we use $f^{(i)}$ to denote the $i$'th formal derivative of $f$. For $i \geq 0$ the numbers $e_i^I(A) = h^{(i)}_I(1)/i!$ are called the \textit{Hilbert coefficients} of $A$. (If $I = \m$ then we drop the superscript $\m$). The number $e_0^I(A)$ is called the \emph{multiplicity} of $A$ \wrt \ $I$. 
As $A$ is \CM, the Hilbert coefficients satisfy various constraints, cf. \cite{Pu1}. There has been a lot of
 work  to understand $G_I(A)$ when the Hilbert coefficients satisfy the boundary values; see \cite{rv}.

Narita proved that $e_2^I(A) \geq 0$, see \cite{Nar}. Furthermore if $\dim A = 2$ and $e_2^I(A) = 0$ then $G_{I^n}(A)$ is \CM \ for all $n \gg 0$. Narita's result is false in dimension $\geq 3$. There are examples of three dimensional \CM \ rings with $e_2^I(A) = 0$ and
$\depth G_{I^n}(A) = 1$ for all $n \gg 0$, see \cite[8.5]{Pu6}. In Theorem \ref{e2} we prove that if $\dim A = 3$ and $e_2^I(A) = 0$ then for all $n \gg 0$ we have$BH^i(G_{I^n}(A)) = 0$ for $i<3$. We prove a similar result when $I$ is integrally closed and $e^I_2(A) = e_1^I(A) - e_0^I(A) + \ell(A/I)$. 

Recall an ideal $I$ is said to be normal if $I^n$ is integrally closed for all
 $n \geq 1$. If $I$ is a normal $\m$-primary ideal the Huckaba and Huneke showed that $\depth G_{I^n}(A) \geq 2$ for all $n \gg 0$, see \cite[Theorem 3.1]{HHu}. Moreover they gave an example of an $\m$-primary normal ideal $I$ in a $3$-dimensional \CM \ local ring with $\depth G_{I^n}(A) = 2$ for all $n \geq 1$, see \cite[3.11]{HHu}. In Theorem \ref{normal}
we prove that if $\dim A \geq 3$ and $I$ is a normal ideal then for all $n \gg 0$ we have $BH^i(G_{I^n}(A)) = 0$ for $i<3$.

Finally we consider the case when $e_2 = e_1 - e_0 + 1$. If $e_2 \leq 2$ the the structure of $G_\m(A)$ is well understood, see  \cite[Section 6]{VaSix}.  In Theorem \ref{chi2} we prove that if $\dim A =3$ and $e_2 = e_1 - e_0 + 1 = 3$ then  $BH^i(G_\m(A)) = 0$ for $i<3$.

Although we are primarily interested in the case for \CM \ rings we prove most of our results for \CM  \ modules $M$. It is technically easier to work with modules. Also note that Bockstein cohomology is a module theoretic construct. So for this reason too it is convenient to work with modules.  Let $G_I(M) = \bigoplus_{n \geq 0}I^nM/I^{n+1}M$ be the associated graded module of $M$ \wrt \ $I$. It can be easily proved that $G_I(M)$ is a finitely generated $G_I(A)$-module.

Here is an overview of the contents of the paper. In section one we recall the notion of Bockstein cohomology and then discuss some properties of it that we need. We give an alternate  construction of Bockstein operators using the $\R(I)$ module $L^I(M) = \bigoplus_{n \geq 0}M/I^{n+1}M$. In section two we discuss some properties of $L^I(M)$ which were proved in \cite{Pu5} and are needed for this paper. In section three we give a condition which characterizes when $BH^0(G_I(M)) = 0$. We also prove a rigidity result for Bockstein cohomology. In section four we discuss the effect on Bockstein operators modulo an element $x^*$ which is $G_I(M)$-regular. We prove an analogue of Sally descent for Bockstein cohomology.  In the next four sections we prove our results. 

\section{Bockstein Cohomology}
In this paper all rings are commutative Noetherian and all modules are assumed to be finitely generated unless specified otherwise.
In this section we first recall a very general construction of Bockstein cohomology. We then specialize to the case of associated graded modules. We then give an alternate description of Bockstein cohomology which is useful for our computations.

\s \emph{General construction of  Bockstein Cohomology}.

 Let $R$ be a ring, $M$ an $R$-module and $x$ a non-zerodivisor on $M$. We have a natural exact sequence
\[
0 \rightarrow \frac{M}{xM} \xrightarrow{\alpha} \frac{M}{x^2M} \xrightarrow{\pi} \frac{M}{xM} \rightarrow 0.
\]
Here $\pi$ is the natural projection map and $\alpha(m + xM) = xm + x^2M$. \\
Let $F \colon Mod(R) \rightarrow Mod(R)$ be any left exact functor. Then note that we have natural maps
\[
\beta^i  \colon RF^i(M/xM) \rightarrow RF^{i+1}(M/xM).
\]
We call $\beta^i$ the $i^{th}$ \emph{ Bockstein operator} on $M/xM$ with respect to $F$.
Consider the natural exact sequence
\[
0 \rt M \xrightarrow{x} M \xrightarrow{\rho} M/xM \rt 0.
\]
So we have an exact sequence
\[
\rt RF^i(M/xM)\xrightarrow{\delta^i} RF^{i+1}(M) \rt RF^{i+1}(M) \xrightarrow{RF^{i+1}(\rho)} RF^{i+1}(M/xM) \rt
\]
It can be easily shown that $\beta^i =  RF^{i+1}(\rho)\circ \delta^i$.  Since $\delta^{i+1} \circ RF^{i+1}(\rho) = 0$ we get
that $\beta^{i+1} \circ \beta^{i} = 0$ for all $i \geq 0$. Thus we have a complex 
\[
\cdots \xrightarrow{\beta^{i-1}} RF^i(M/xM) \xrightarrow{\beta^{i}} RF^{i+1}(M/xM) \xrightarrow{\beta^{i+1}} RF^{i+2}(M/xM) \cdots
\]
The cohomology of this complex is denoted by $BF^*(M/xM)$ and is called the \emph{Bockstein cohomology} of $M/xM$ with respect to $F$.

\s \textit{Bockstein Cohomology of Associated graded modules }

Let $\R(I) = \bigoplus_{n \in \mathbb{Z}}I^n$ be the extended Rees-ring of $A$ with respect to $I$.  Here $I^n = A$ for all $n \leq 0$ and $\R(I)$ is considered as a subring of $A[t,t^{-1}]$. Let $\R(I)_+$ to be the ideal in $\R(I)$ generated by $ \bigoplus_{n >0}I^n$.
Let $M$ be an $A$-module. Let $\R(I,M) = \bigoplus_{n \in \mathbb{Z}}I^nM$ be the extended Rees-module of $M$ with respect to $I$. 

Clearly $t^{-1}$ is a non-zero divisor on $\R(I,M)$. Note $\R(I,M)/t^{-1}\R(I,M) = G_I(M)$. We have an exact sequence (after a shift)
\[
0 \rightarrow G_I(M) \rightarrow \R(I,M)/t^{-2}\R(I,M)(-1) \rightarrow G_I(M)(-1) \rightarrow 0.
\]
Here 
$$\frac{\R(I,M)}{t^{-2}\R(I,M)} = M/IM \oplus M/I^2M \oplus IM/I^3M\oplus I^2M/I^4M \oplus \cdots \oplus I^{n-1}M/I^{n+1}M \cdots,$$
with $M/IM$ sitting in degree $-1$.

Let $\Gamma_{\R(I)_+} \colon Mod(\R(I)) \rightarrow Mod(\R(I))$ be the $\R(I)_+$-torsion functor. So by the general theory we have Bockstein homomorphisms
\[
\beta^i \colon H^i_{G_+}(G_I(A))(-1) \rightarrow H^{i+1}_{G_+}(G_I(A)),
\]
and we have Bockstein cohomology modules
$$BH^i_{G_+}(G_I(A))  = \ker(\beta^{i}(+1))/\image(\beta^{i-1})   \quad \text{ for all $i \geq 0$}. $$
Set $\beta^i_I(M) = \beta^i(G_I(M))$.

\begin{remark}
Let $(A,\m) \rt (A^\prime,\m^\prime)$ be a flat extension with $\m A^\prime = \m^\prime$. Set $I^\prime = IA^\prime$ and $M^\prime = M \otimes_A A^\prime$. Then it is clear that 
$$\beta^i_{I^\prime}(M^\prime)  = \beta^i_I(M)\otimes_A A^\prime.$$
It follows that for all $i \geq 0$ we have 
\[
BH^i_{G^\prime_+}(G_{I^\prime}(A^\prime)) \cong   BH^i_{G_+}(G_I(A))\otimes_A A^\prime.
\]
We use this primarily when the residue field $k$ of $A$ is finite. In this case we set $A^\prime = A[X]_{\m A[X]}$. Note that the residue field of $A^\prime$ is $k(X)$ which is infinite. Thus for many computations we may assume that the residue field of $A$ is infinite. 
\end{remark}

\s Although for definition of Bockstein cohomology we used the extended Rees algebra, for computation it is easier to use the following $\R(I)$-module:
$$ L^I (M)= \bigoplus_{n \geq 0} \frac{M}{I^{n+1}M}.$$
To see that $L^I(M)$ is an $\R(I)$-module, note that we have an exact sequence 
\[
0 \rt \R(I,M) \rt M[t,t^{-1}] \rt L^I(M)(-1) \rt 0.
\]
By this exact sequence we can give $L^I(M)$ a structure of $\R(I)$-module.
Note that $L^I(M)$ is \textit{not} finitely generated as a $\R(I)$-module. For $r \geq 0$ consider the finitely generated 
submodules $L^I_r(M) $ of $L^I(M)$ defined as follows:
\[
L^I_r (M)= \left<\bigoplus_{n =0}^{r} \frac{M}{I^{n+1}M} \right>.
\]
Notice that $L^I_0(M) = G_I(M)$ and $L^I_1(M) = \R(I,M)/t^{-2}\R(I,M)(-1)$.

\s \textit{Definition of Bockstein cohomology via $L^I_r(M)$}

Set $L_r = L_r^I(M)$, $L = L^I(M)$ and $G = G_I(M)$.  For systemic reasons set $L_{-1} = 0$. For all $r \geq 0$ we have an exact sequence
\begin{equation}\label{Lr-eq}
  0 \rightarrow L_{r-1} \rightarrow L_r \xrightarrow{\rho_r} G(-r) \rightarrow 0. 
\end{equation}
For $r = 1$ we get
\[
0 \rightarrow G \rightarrow L_1 \rightarrow G(-1)\rightarrow 0.
\]
This is nothing but the defining exact sequence for Bockstein cohomology. For $r \geq 0$ we first
take local cohomology   of the exact sequence (\ref{Lr-eq}) for $r+1$ \wrt \ $\R(I)_+$. We obtain
\begin{equation}\label{alt-defn-1}
\cdots H^i(L_{r+1}) \rightarrow H^i(G(-r-1)) \xrightarrow{\delta^i_{r+1}} H^{i+1}(L_r) \cdots
\end{equation}
Taking local cohomology of the  exact sequence (\ref{Lr-eq}) for $r$ we obtain
\begin{equation}\label{alt-defn-2}
\cdots \rt H^{i+1}(L_{r}) \xrightarrow{H^{i+1}(\rho_r)} H^{i+1}(G(-r)) \rt  \cdots
\end{equation}
So we obtain maps $\alpha^i_r =  H^{i+1}(\rho_r)\circ \delta^i_{r+1} \colon  H^i(G(-r-1)) \rightarrow H^{i+1}(G(-r))$.
Notice that $\alpha^i_0 = \beta^i$. More generally we have

\begin{proposition}\label{Lr}
(with hypotheses as above)
 $$\beta^i = \alpha^i_r(r).$$
\end{proposition}
\begin{proof}
For $r \geq 0$ we have an exact sequence 
\[
0 \rt G \rt L_r \xrightarrow{\pi_r} L_{r-1}(-1) \rt 0.
\]
Consider the commutative diagram $\mathcal{C}_r$ with exact rows
\[
  \xymatrix
{
 0
 \ar@{->}[r]
  & L_r
    \ar@{->}[d]^{\pi_r}
\ar@{->}[r]
 & L_{r+1}
    \ar@{->}[d]^{\pi_{r+1}}
\ar@{->}[r]
& G(-r-1)
    \ar@{->}[d]^{\xi}
\ar@{->}[r]
 &0
 \\
 0
 \ar@{->}[r]
  & L_{r-1}(-1)
\ar@{->}[r]
 & L_r(-1)
\ar@{->}[r]
& G(-r)(-1)
\ar@{->}[r]
&0
 }
\]
It can be easily shown that $\xi$ is the identity map. So we have a commutative diagram
\[
\xymatrix
{
& H^i(G)(-r-1)
\ar@{->}[r]
\ar@{->}[d]^{H^{i}(\xi)}
& H^{i+1}(L^r)
\ar@{->}[d]^{H^{i+1}(\pi_r)}
\\
&H^i(G(-r))(-1)
\ar@{->}[r]
& H^{i+1}(L^{r-1})(-1)
}
\]
By considering the diagram $\mathcal{C}_{r-1}$ we obtain a commutative diagram
\[
\xymatrix
{
& H^{i+1}(L_r)
\ar@{->}[r]
\ar@{->}[d]^{H^{i+1}(\pi_r)}
& H^{i+1}(G(-r))
\ar@{->}[d]^{H^{i+1}(\xi)}
\\
& H^{i+1}(L_{r-1})(-1)
\ar@{->}[r]
& H^{i+1}(G(-r+1))(-1)
}
\]
Since $H^i(\xi)$  and $H^{i+1}(\xi)$  are identity maps we get that $\alpha^i_r = \alpha_{r-1}^i(-1)$.  So $\alpha^i_r(r) = \alpha^i_{r-1}(r-1)$. Therefore we obtain that
$\alpha^i_r(r) = \alpha^i_0(0)= \beta^i.$
\end{proof}

 \section{Some Properties of $L^{I}(M)$}\label{Lprop}

In this section we collect some of  the properties of $L^{I}(M)$ which we proved in \cite{Pu5}. Throughout this section
$(A,\m)$ is a  local ring with infinite residue field, $M$ is a \emph{\CM }\ module of dimension $r \geq 1$ and $I$ is \emph{an ideal of definition} for
$M$, i.e., $\ell(M/IM)$ is finite.

\s \label{mod-struc} Set $\Sc(I) = A[It]$;  the Rees Algebra of $I$. In \cite[4.2]{Pu5} we proved that \\
$L^{I}(M) = \bigoplus_{n\geq 0}M/I^{n}M$ is a $\Sc(I)$-module. Note that we also gave $L^I(M)$ an $\R(I)$-module structure and as $\Sc(I)$ is a subring of $\R(I)$ we have an induced $\Sc$-module structure on $L^I(M)$. It is easily verified that these two $\Sc(I)$-module structures on $L^I(M)$ are the same. 

 \s Set $\M = \M_{\Sc(I)} = \m\oplus \Sc(I)_+$.  In \cite{Pu5} we proved many properties of \\  $H^i_{\M}(L^I(M))$. In this paper we need properties of the local cohomology modules $H^i_{\R(I)_+}(L^I(M))$. Note that $H^i_{\R(I)_+}(L^I(M)) = H^i_{\Sc(I)_+}(L^I(M))$ for all $i \geq 0$.
 For all $i \geq 0$ we also have  natural maps $\theta^{(i)} \colon H^i_{\M}(L^I(M)) \rt 
H^i_{\Sc_+}(L^I(M))$. Our first result is

\begin{proposition}
[with hypotheses as above] For every $i \geq 0$ the map $\theta^{(i)}$ is an isomorphism.
\end{proposition}
\begin{proof}
Let $L = L^I(M)$ and for $r \geq 0 $ let $L_r = L_r^I(M)$. Notice
\[
L = \bigcup_{r \geq 0} L_r.
\]
It follows that for every $i \geq 0$ we have an isomorphism
\[
H^i_{\q}(L)  = \varinjlim H^i_{\q}(L_r) \quad \text{where} \ \q \in \{ S_+, \M \}; \ \text{see \cite[3.4.10]{BSh}}.
\]
Note that $L_r$ is a finitely generated $\Sc$-module with each component of finite length. It is elementary fact that in this case the natural maps $\theta^{(i)}_r \colon H^i_{\M}(L_r) \rt 
H^i_{\Sc(I)_+}(L_r)$ are isomorphisms. It is also clear that for all $i \geq 0$ we have
$ \theta^{(i)} = \varinjlim \theta^{(i)}_r$. It follows that $\theta^{(i)}$ is an isomorphism for all $i \geq 0$.
\end{proof}

\s Let $H^{i}(-) = H^{i}_{\M}$ denote the $i^{th}$-local cohomology functor \wrt \ $\M$.
 Recall a graded $\Sc(I)$-module $V$ is said to be
\textit{*-Artinian} if
every descending chain of graded submodules of $V$ terminates. For example if $E$ is a finitely generated $\Sc(I)$-module then $H^{i}(E)$ is *-Artinian for all
$i \geq 0$.

\s \label{zero-lc} In \cite[4.7]{Pu5} we proved that
\[
H^{0}(L^I(M)) = \bigoplus_{n\geq 0} \frac{\wt{I^{n+1}M}}{I^{n+1}M}.
\]
Here $\widetilde{KM}$ denotes the Ratliff-Rush closure of $M$ \wrt \ an ideal $K$.  Recall
\[
\widetilde{KM} = \bigcup_{i \geq 1}K^{i+1}M \colon K^i.
\] 
\s \label{Artin}
For $L^I(M)$ we proved that for $0 \leq i \leq  r - 1$
\begin{enumerate}[\rm (a)]
\item
$H^{i}(L^I(M))$ are  *-Artinian; see \cite[4.4]{Pu5}.
\item
$H^{i}(L^I(M))_n = 0$ for all $n \gg 0$; see \cite[1.10 ]{Pu5}.
\item
 $H^{i}(L^I(M))_n$  has finite length
for all $n \in \mathbb{Z}$; see \cite[6.4]{Pu5}.
\end{enumerate}

\s \label{I-FES} The natural maps $0\rt I^nM/I^{n+1}M \rt M/I^{n+1}M \rt M/I^nM \rt 0 $ induce an exact
sequence of $\Sc(I)$-modules
\begin{equation}
\label{dag}
0 \xar G_{I}(M) \xar L^I(M) \xrightarrow{\Pi} L^I(M)(-1) \xar 0.
\end{equation}
We call (\ref{dag}) \emph{the first fundamental exact sequence}.  We use (\ref{dag}) also to relate the local cohomology of $G_I(M)$ and $L^I(M)$.

\s \label{II-FES} Let $x$ be  $M$-superficial \wrt \ $I$, i.e., $(I^{n+1}M \colon x) = I^nM$ for all $n \gg 0$. Set  $N = M/xM$ and $u =xt \in \Sc(I)_1$. Notice $L^I(M)/u L^I(M) = L^I(N)$.  
For each $n \geq 1$ we have the following exact sequence of $A$-modules:
\begin{align*}
0 \xar \frac{I^{n+1}M\colon x}{I^nM} \xar \frac{M}{I^nM} &\xrightarrow{\psi_n} \frac{M}{I^{n+1}M} \xar \frac{N}{I^{n+1}N} \xar 0, \\
\text{where} \quad \psi_n(m + I^nM) &= xm + I^{n+1}M.
\end{align*}
This sequence induces the following  exact sequence of $\Sc$-modules:
\begin{equation}
\label{dagg}
0 \xar \B^{I}(x,M) \xar L^{I}(M)(-1)\xrightarrow{\Psi_u} L^{I}(M) \xrightarrow{\rho^x}  L^{I}(N)\xar 0,
\end{equation}
where $\Psi_u$ is left multiplication by $u$ and
\[
\B^{I}(x,M) = \bigoplus_{n \geq 0}\frac{(I^{n+1}M\colon_M x)}{I^nM}.
\]
We call (\ref{dagg}) the \emph{second fundamental exact sequence. }

\s \label{long-mod} Notice  $\lambda\left(\B^{I}(x,M) \right) < \infty$. A standard trick yields the following long exact sequence connecting
the local cohomology of $L^I(M)$ and
$L^I(N)$:
\begin{equation}
\label{longH}
\begin{split}
0 \xar \B^{I}(x,M) &\xar H^{0}(L^{I}(M))(-1) \xar H^{0}(L^{I}(M)) \xar H^{0}(L^{I}(N)) \\
                  &\xar H^{1}(L^{I}(M))(-1) \xar H^{1}(L^{I}(M)) \xar H^{1}(L^{I}(N)) \\
                 & \cdots \cdots \\
               \end{split}
\end{equation}

\s \label{Artin-vanish} We will use the following well-known result regarding *-Artinian modules quite often:

Let $L$ be a *-Artinian $\Sc(I)$-module.
\begin{enumerate}[\rm (a)]
\item
If $\psi \colon L(-1) \rt L$ is a monomorphism then $L = 0$.
\item
If $\phi \colon L \rt L(-1)$ is a monomorphism then $L = 0$.
\end{enumerate}

\s\label{Veronese} One huge advantage of considering $L^I(M)$ is that it behaves well \wrt \ the Veronese functor. Notice
\[
\left(L^I(M)(-1)\right)^{<t>} = L^{I^t}(M)(-1) \quad \text{for all} \ t \geq 1.
\]
Also note that $\Sc(I)^{<t>} = \Sc(I^t)$ and that $(\M_{\Sc(I)})^{<t>} = \M_{\Sc(I^t)}$. It follows that for all $i \geq 0$
\[
\left(H^i_{\M_{\Sc(I)}}(L^I(M)(-1)\right)^{<t>} \cong H^i_{\M_{\Sc(I^t)}}(L^{I^t}(M)(-1).
\]
By \ref{Artin}(b) it follows that for for $t \gg 0$ we have $H^0(L^{I^t}(M)) = 0$ and for  $1 \leq i \leq r-1$ we have
\[
H^i(L^{I^t}(M))_{n} = 0 \quad \text{for} \ n \geq 0.
\]

\s \label{power-of-I} Set
 \begin{align*}
\xi_I(M) &:= \underset{0 \leq i \leq r-1}\min\{ \ i \ \mid H^{i}(L)_{-1} \neq 0 \ \text{or} \ \ell(H^{i}(L)) = \infty \}.\\
\amp_I(M) &:= \max\{\ |n| \ \mid H^{i}(L)_{n-1} \neq 0 \  \text{for} \ i = 0,\ldots, \xi_I(M) - 1 \}.
\end{align*}
In \cite[7.5]{Pu5} we showed that
\[
 \depth G_{I^l}(M) = \xi_I(M) \ \text{for all} \ l > \amp_I(M).
\]

\s \label{xi}
By Proposition 9.2 in \cite{Pu5} and its proof  it follows that the following conditions are equivalent:
\begin{enumerate}
\item
$\xi_I(M) \geq 2$
\item
$H^1(L^I(M))_n = 0$ for $n < 0$.
\item
$H^1(L^I(M))_{-1} = 0$.
\item
$H^1(G_I(M))_n = 0 $ for $n < 0$.
\item
$H^1(G_I(M))_{-1} = 0 $.
\end{enumerate}

\section{Vanishing of $BH^0(G)$}
Let $M$ be a Cohen-Macaulay module of dimension $\geq 1$ and  let $I$ be an ideal of definition of $M$. Set $G = G_I(M)$. In this section we characterize when $BH^0(G) = 0$. 
We also prove that if $\dim M \geq 2$ and if $BH^1(G) = 0$ then $BH^0(G) = 0$. 

\s 
It is well-known that $H^0(G) = 0$ if and only if $\widetilde{I^nM} = I^nM$ for all
$n \geq 1$. 

For Bockstein cohomology we have the following result.
\begin{proposition}\label{zero-b}
(with hypotheses as above)
$$BH^0(G) = 0  \Longleftrightarrow \quad \widetilde{I^{j+1}M} \subseteq I^jM  \quad  \forall j \geq 1. $$
\end{proposition}
\begin{proof}
Set $L = L^I(M)$ and for $r \geq 0$ set $L_r = L^I_r(M)$.
If $ \widetilde{I^{j+1}M} \subseteq I^jM  \quad  \forall j \geq 1$ then note that $H^0(G) =  H^0(L)$.  Notice that for all $r \geq 0$ we have
$$H^0(G) \subseteq H^0(L_r) \subseteq H^0(L).$$
So   $H^0(L_r) = H^0(G)$ for all $r \geq 0$. In particular $H_0(L_1) = H_0(G)$. Consider the exact sequence
\[
0 \rt G \rt L_1 \rt G(-1).
\]
Computing the long exact sequence in cohomology we get that  $\beta^0 \colon H^0(G)(-1) \rt H^1(G)$ is injective. So $BH^0(G) = 0$.

Conversely if $BH^0(G) = 0$ we have that $\beta^0$ is injective. By Proposition \ref{Lr} we get that $\alpha_r$ is injective for all $r \geq 0$.
It follows that for all $r \geq 0$ the natural inclusion  $H^0(L_r) \rt H^0(L_{r+1})$ is an isomorphism. So 
we have that 
\[
H^0(G) = H^0(L_0) = H^0(L_r) \quad \text{for all}\  r \geq 0.
\] 
As $H^0(L)$  has finite length it follows that $H^0(L) = H^0(L_r)$ for all $r \gg 0$. Therefore $H^0(G) = H^0(L)$. Fix $j \geq 1$. As $H^0(G)_j = H^0(L)_j$ we have that
\[
\frac{\widetilde{I^{j+1}M}\cap I^jM}{I^{j+1}M}   = \frac{\widetilde{I^{j+1}M}}{I^{j+1}M}.  
\]
It follows that $ \widetilde{I^{j+1}M} \subseteq I^jM$.
\end{proof}
We now prove a rigidity result for Bockstein cohomology. 
\begin{theorem}\label{rigid}(with hypotheses as above). Assume $\dim M \geq 2$.
If
$BH^1(G) = 0$  then  $BH^0(G) = 0.$
\end{theorem}
\begin{proof}
Set $L = L^I(M)$ and for $r \geq 0$ set $L_r = L^I_r(M)$
Using  equation (\ref{alt-defn-2}) it follows that 
$$\image \alpha^i_r \subseteq \image(H^{i+1}(\rho_r))  =  \ker \delta^{i+1}_r  \subseteq \ker \alpha^{i+1}_{r-1}. $$
As $BH^1(G) = 0$ we have that $\ker \beta^1(1) = \image \beta^0$. It follows that for all $r \geq 0$ we have
$\image \alpha^0_r = \ker \alpha^{1}_{r-1}$. Note that there is no shift in the later equation. So we have $\image \alpha^0_r = \image H^1(\rho_r)$ for all $r \geq 0$. 

As $H^0(L)$ has finite length it follows that $H^0(L_r) = H^0(L)$ for all $r \gg 0$, say from $r \geq c$.
Fix $r \geq c+1$. Note that we have a commutative diagram
\[
  \xymatrix
{
 0
 \ar@{->}[r]
  & L_{r-1}
    \ar@{->}[d]^{i}
\ar@{->}[r]
 & L
    \ar@{->}[d]^{\xi}
\ar@{->}[r]
& L(-r)
    \ar@{->}[d]^{\pi(-r)}
\ar@{->}[r]
 &0
 \\
 0
 \ar@{->}[r]
  & L_{r}
\ar@{->}[r]
 & L
\ar@{->}[r]
& L(-r-1)
\ar@{->}[r]
&0
 }
\]
Here $\xi$ is the identity map and $\ker \pi(-r) = \coker i = G(-r)$. Taking cohomology and as $H^0(L_{r-1}) = H^0(L_r) = H^0(L)$ we have
a diagram
\[
  \xymatrix
{
 0
 \ar@{->}[r]
  &H^0(L )(-r)
    \ar@{->}[d]^{H^0(\pi(-r))}
\ar@{->}[r]
 & H^1(L_{r-1})
    \ar@{->}[d]^{H^1(i)}
\ar@{->}[r]
& H^1(L)
    \ar@{->}[d]^{H^1(\xi)}
 \\
 0
 \ar@{->}[r]
  & H^0(L)(r-1)
\ar@{->}[r]
 & H^1(L_r)
\ar@{->}[r]
& H^1(L)
 }
\]
Note that $\ker H^0(\pi(-r)) = H^0(G)(-r)$.  Note $H^1(\xi)$ is the identity map. Further note that $K = \coker H^1(i) = \image H^1(\rho_r) = \image \alpha^0_r$. Set
$C = \coker H^0(\pi(-r))$.

Note that we have an induced map $\theta \colon C \rt K$. Since $H^1(\xi)$ is the identity map, a simple diagram chase shows that $\theta $ is injective.
Note that $\ell(C) = \ell (H^0(G))$.  However $\ell(K) = \ell(\image \alpha^0_r) \leq \ell(H^0(G))$. It follows that $\ell(\image \alpha^0_r) = \ell(H^0(G))$. It follows that
$\alpha^0_r$ is injective. So $\beta^0$ is injective. Thus $BH^0(G) = 0$.
\end{proof}
We also have the following very general result on the vanishing of Bockstein operator.
\begin{proposition}\label{vanishing}[with hypotheses as above] 
\[
H^i(L) = 0 \implies \beta^i = 0.
\]
\end{proposition}
\begin{proof}
For all $r \geq 0$ we  have an exact sequence
\begin{equation}\label{Lr-L-eq}
0 \rightarrow L_r \rightarrow L \rightarrow L(-r-1) \rightarrow 0. 
\end{equation}
Also note that we have a commutative diagram with exact rows
\[
  \xymatrix
{
 0
 \ar@{->}[r]
  & L_{r-1}
    \ar@{->}[d]^{j}
\ar@{->}[r]
 & L
    \ar@{->}[d]^{\xi}
\ar@{->}[r]
& L(-r)
    \ar@{->}[d]^{\pi(-r)}
\ar@{->}[r]
 &0
 \\
 0
 \ar@{->}[r]
  & L_{r}
\ar@{->}[r]
 & L
\ar@{->}[r]
& L(-r-1)
\ar@{->}[r]
&0
 }
\]
Here $\xi$ is the identity map and $\ker \pi(-r) = \coker j = G(-r)$. So we have a commutative diagram
\[
\xymatrix
{
\ 
 \ar@{->}[r]
  & H^{i}(L)
    \ar@{->}[d]^{H^{i}(\pi(-r))}
\ar@{->}[r]
 & H^{i+1}(L_{r-1})
    \ar@{->}[d]^{H^{i+1}(j)}
\ar@{->}[r]
& H^{i+1}(L)
    \ar@{->}[d]^{H^{i+1}(\xi)}
\ar@{->}[r]
 & \
 \\
 \
 \ar@{->}[r]
  & H^{i}(L)
\ar@{->}[r]
 &  H^{i+1}(L_{r})
\ar@{->}[r]
& H^{i+1}(L)
\ar@{->}[r]
& \ 
 }
\]
If $H^i(L)$ is zero then as $H^{i+1}(\xi)$ is the identity map it follows that  the natural map $H^{i+1}(j) \colon H^{i+1}(L_{r-1}) \rt H^{i+1}(L_r)$ is an inclusion. So  $\delta^i_{r}$ in \ref{alt-defn-1} is the zero map. Therefore $\alpha^i_{r-1} = 0$. Thus $\beta^i = 0$.
\end{proof}
\begin{example}
Let $(A,\m)$ be local and let $N$ be a $2$-dimensional \CM \ $A$-module. Let $I$ be an ideal of definition for $N$. Assume that any one of the following conditions hold
\begin{enumerate}
\item
$e_2^I(N) = 0$.
\item
$N = A$, the ideal $I$ is an integrally closed   with $e_2^I(A) = e_1^I(A)- e_0^I(A) + \ell(A/I).$
\end{enumerate}
Then $BH^i(G_I(N)) = 0$ for $i = 0, 1$.

By \cite[4.3]{Pu6} and \cite[4.4]{Pu6} we have that $\widetilde{I^{i+1}N} \subseteq I^iN$ for all $i \geq 1$. It follows that $BH^0(G_I(N)) = 0$, i.e., $\beta^0$ is injective.

By  \cite[4.5]{Pu6} $H^1(L^I(N)) = 0$. It follows that $\beta^1 = 0$. Since $H^0(G_I(N) \cong H^0(L^I(M))$ it follows from \ref{dag} that $H^1(G_I(N)) \cong H^0(G_I(N)(-1)$. Therefore $\beta^0$ is an isomorphism. It follows that 
$BH^1(G_I(N)) = 0$.
\end{example}

We now give an example of a one dimensional \CM \ local ring $A$ and an $\m$-primary ideal with  $BH^0(G_I(A)) \neq 0$.
The  example is from \cite[1.18]{HLS}.
\begin{example}\label{HLS-ex}
Let $A = k[[ t^5, t^6]]$. Let $I = (t^{10}, t^{11})$. Then note that $t^{24} \notin I$.
However it can be easily verified that $t^{24} \in (I^4 \colon I^2) \subseteq \widetilde{I^2}$. Thus we have $\widetilde{I^2} \nsubseteq I$. So by \ref{zero-b} we have that $BH^0(G_I(A)) \neq 0$.
\end{example}

\begin{remark}\label{zeroTOn}
Let $(A,\m)$ be a one dimensional \CM \ local ring and let $I$ be an $\m$-primary ideal with $BH^0(G_I(A)) \neq 0$. Consider $B = A[X_1,\cdots,X_n]_{\n}$ where 
$\n = (\m, X_1,\ldots, X_n)$. Set $J = (I, X_1,\ldots,X_n)$. Clearly $B$ is a \CM \ local ring of dimension $n+1$ and $J$ is a $\n B$-primary ideal. In \ref{vashi2} we prove that
$BH^n(G_J(B)) \neq 0$.
\end{remark}

\begin{remark}
Example \ref{HLS-ex} is rather simple. However the author does not know of an example of a monomial ideal $I$ in $k[X_1,\ldots,X_n]$ (here $2 \leq n \leq 4$) with $\widetilde{I^2} \nsubseteq I$.
\end{remark}
\section{Bockstein operators modulo a super regular  element \\ and Sally Descent for Bockstein Cohomology}
Let $M$ be a \CM \ $A$-module of dimension $r \geq 2$ and let $I$ be an ideal of definition for $M$. Let $x \in I\setminus I^2$ be such that $x^*$ is $G_I(M)$ regular. Here $x^*$ is the image of $x$ in $I/I^2$. Set
$N = M/xM$. In this section we relate the Bockstein operators of $N$ and $M$. This will be used in the later sections. We also prove an analogue for Sally descent for Bockstein
cohomology.

\s \label{setup}  Set $G = G_I(M), \ov{G} = G_I(N), \R = \R(I,M)$ and $\ov{\R} = \R(I,N)$. Set $u = xt \in \R(I)_1$. Note that $u$ is $\R$-regular and $\R/(u) = \ov{\R}$.  Furthermore notice that
the action of $u$ on $G$ is same as that of $x^*$. It follows that $t^{-1},u$ is a $\R$-regular sequence. So  $t^{-2}, u$ is also a $\R$-regular sequence, cf., \cite[16.1]{Mat}.  Therefore we get a commutative diagram
\[
  \xymatrix
{
\
&0
\ar@{->}[d]
&0
\ar@{->}[d]
&0
\ar@{->}[d]
\
\\
 0
 \ar@{->}[r]
  & G(-1)
\ar@{->}[r]^{x^*}
\ar@{->}[d]
 & G
\ar@{->}[r]
\ar@{->}[d]
& \ov{G}
\ar@{->}[r]
\ar@{->}[d]
&0
\\
 0
 \ar@{->}[r]
  &\R/t^{-2}\R(-1)
    \ar@{->}[d]
\ar@{->}[r]^{u}
 & \R/t^{-2}\R
    \ar@{->}[d]
\ar@{->}[r]
& \ov{\R}/t^{-2}\ov{R}
    \ar@{->}[d]
    \ar@{->}[r]
    &0
 \\
 0
 \ar@{->}[r]
  & G(-2)
\ar@{->}[r]^{x^*}
\ar@{->}[d]
 & G(-1)
\ar@{->}[r]
\ar@{->}[d]
& \ov{G}(-1)
\ar@{->}[r]
\ar@{->}[d]
&0
\\
\
&0
&0
&0
\
 }
\]
As a corollary we obtain
\s \label{mod}(with hypotheses as in \ref{setup}) We have a commutative diagram
 \[
  \xymatrix
{
  &H^i(\ov{G})(-1)
    \ar@{->}[d]^{\ov{\beta^i}}
\ar@{->}[r]
 & H^{i+1}(G)(-2)
    \ar@{->}[d]^{\beta^{i+1}(-1)}
\ar@{->}[r]^{x^*}
& H^{i+1}(G)(-1)
    \ar@{->}[d]^{\beta^{i+1}}
 \\
  & H^{i+1}(\ov{G})
\ar@{->}[r]
 & H^{i+2}(G)(-1)
\ar@{->}[r]^{x^*}
& H^{i+2}(G))
 }
\]

Sally descent is a basic technique in our area.  Let $x \in I$ be $M$-superficial \wrt\ $I$. Then Sally descent says that
$$\depth G_{I}(M/xM) \geq r \geq 1 \iff \depth G_I(M) \geq r+1. $$
We now prove a version of Sally descent for Bockstein cohomology modules.  Unfortunately the hypothesis is more restrictive. However in section 7
we will use this result. We will use Matlis duality in the proof of the theorem. So for convenience we take
$I$ to be $\m$-primary and not just an ideal of definition for $M$. Also let $(-)^\vee$ denote the Matlis dual of a $G_I(A)$-module.

\begin{theorem}\label{Sally-descent}
Let $M$ be a \CM \ $A$-module and let $I$ be an $\m$-primary ideal.
Let $x \in I\setminus I^2$ be such that $x^*$ is $G_I(M)$ regular. Also assume that for $i =1,\ldots,r$ either $H^i(G_I(M))$ is zero or $x^*$ is $H^i(G_I(M))^\vee$ regular. Let $r< \dim M$.
Then
$$BH^i\left(G_{I}(M/xM)\right) = 0 \ \text{for} \ i \leq r-1 \iff BH^i(G_I(M)) = 0  \ \text{for} \ i \leq r. $$
\end{theorem}
\begin{proof}
Set $G = G_I(M)$ and $\ov{G} = G_I(M/xM)$. As $H^0(G) = 0$ we get that $BH^0(G) = 0$. 

We prove by induction on $m$ with $1 \leq m \leq r$ that
\begin{enumerate}
\item
$BH^m(G) = 0$ $\iff BH^{m-1}(\ov{G}) = 0$.
\item
We have an exact sequence
$$0 \rt \frac{H^{m}(\ov{G})}{\ov{\beta^{m-1}}\left(H^{m-1}(\ov{G})\right)}  \rt \frac{H^{m+1}(G)}{\beta^{m}\left(H^m(G)\right)}(-1) \xrightarrow{x^*} \frac{H^{m+1}(G)}{\beta^{m}\left(H^m(G)\right)}. $$
Furthermore the multiplication by $x^*$ is surjective if $m <r$.
\end{enumerate}
We first prove the result for $m = 1$. Note by \ref{mod} we have an commutative diagram
\[
  \xymatrix
{    
0
 \ar@{->}[r] 
  &H^0(\ov{G})(-1)
    \ar@{->}[d]^{\ov{\beta^0}}
\ar@{->}[r]
 & H^{1}(G)(-2)
    \ar@{->}[d]^{\beta^{1}(-1)}
\ar@{->}[r]^{x^*}
& H^{1}(G)(-1)
    \ar@{->}[d]^{\beta^{1}}
    \ar@{->}[r]  
  & 0 
 \\
 \
 \ar@{->}[r]  
  & H^{1}(\ov{G})
\ar@{->}[r]^{\delta^1}
 & H^{2}(G)(-1)
\ar@{->}[r]^{x^*}
& H^{2}(G)
& \ 
 }
\]
Note that multiplication by $x^*$ on $H^1(G)$ is surjective since $x^*$ is $H^1(G)^\vee$-regular. It also follows that
$\delta^1$ is injective. Thus the diagram above satisfies the hypotheses of Snake Lemma.
By Snake Lemma and as $\beta^0 = 0$ we have an exact sequence
\[
0 \rt BH^0(\ov{G}) \rt BH^1(G)(-1) \rt BH^1(G).
\]
If $BH^0(\ov{G}) = 0$ then we have an inclusion $BH^1(G)(-1) \rt BH^1(G)$. As $H^1(G)$ is $*$-Artinian $G_I(A)$-module we have that its subquotient $BH^1(G)$ is also $*$-Artinian. 
By \ref{Artin-vanish}  it follows that $BH^1(G) = 0$. Conversely if $BH^1(G) = 0$ then by the above exact sequence we get $BH^0(\ov{G}) = 0$.

We now assume $BH^1(G) = 0$.
By Snake Lemma  we have an exact sequence
\[
0 \rt \frac{H^{1}(\ov{G})}{\ov{\beta^{0}}\left(H^{0}(\ov{G})\right)}  \rt \frac{H^{2}(G)}{\beta^{1}\left(H^1(G)\right)}(-1) \xrightarrow{x^*} \frac{H^{2}(G)}{\beta^{1}\left(H^1(G)\right)}. 
\]
Note that if $2 \leq r$ then multiplication by $x^*$ on $H^2(G)$ is surjective since $x^*$ is $H^2(G)^\vee$-regular.
It again follows by the Snake Lemma that the map 
$$\frac{H^{2}(G)}{\beta^{1}\left(H^1(G)\right)}(-1) \xrightarrow{x^*} \frac{H^{2}(G)}{\beta^{1}\left(H^1(G)\right)}$$
is surjective.

Thus we have proved the assertion for $m = 1$. Assume the result for $m = i$ and we prove the  result for $m = i+1$ (if $i < r$). Note as $\beta^j \circ \beta^{j-1} = 0$, by \ref{mod} and the assertion (2) of our inductive hypotheses we have a commutative diagram
\[
  \xymatrix
{    
0
 \ar@{->}[r] 
  & \frac{H^{i}(\ov{G})}{\ov{\beta^{i-1}}\left(H^{i-1}(\ov{G})\right)}(-1) 
    \ar@{->}[d]^{\ov{\beta^i}}
\ar@{->}[r]
 & \frac{H^{i+1}(G)}{\beta^{i}\left(H^i(G)\right)}(-2)
    \ar@{->}[d]^{\beta^{i+1}(-1)}
\ar@{->}[r]^{x^*}
& \frac{H^{i+1}(G)}{\beta^{i}\left(H^i(G)\right)}(-1)
    \ar@{->}[d]^{\beta^{i+1}}
    \ar@{->}[r]  
  & 0 
 \\
 \
 \ar@{->}[r]  
  & H^{i+1}(\ov{G})
\ar@{->}[r]^{\delta^{i+1}}
 & H^{i+2}(G)(-1)
\ar@{->}[r]^{x^*}
& H^{i+2}(G)
& \ 
 }
\]
As $x^*$ is $H^{i+1}(G)^\vee$-regular we get that multiplication by $x^*$ on $H^{i+1}(G)$ is surjective. It follows that $\delta^{i+1}$ is injective. Thus we can apply the Snake Lemma again. By an argument similar to the case $m =1$ we can prove the assertion for $m = i+1$. 
\end{proof}
We now give an example where the hypotheses of our Theorem on Sally descent is satisfied.
\begin{example}\label{Variable}
Let $(A,\m)$ be a \CM \ local ring of dimension $d \geq 1$ and let $J$ be an $\m$-primary ideal. Let $B = A[X]_\n$ where $\n = (\m, X)$. Set $I = (J, X)$. Then $J$ is $\n B$-primary. Furthermore $G_I(B) \cong G_J(A)[X^*]$. It can be easily verified that in this case $H^i(G_I(B))^\vee \cong H^{i-1}(G_J(A))^\vee[X^*](-1)$ for $i \geq 1$. Thus in this case the hypotheses of Theorem \ref{Sally-descent} are satisfied. 
\end{example}
We give two applications of the above example.
\begin{example}(with hypotheses as in \ref{Variable}).
If $BH^i(G_J(A)) = 0 $ for $i \leq r-1$ then $BH^i(G_I(B)) = 0 $ for $i \leq r$.
\end{example}

\begin{example} \label{vashi}((with hypotheses as in \ref{Variable}).
If $BH^i(G_J(A)) = 0 $ for $i \leq r-1$ and $BH^r(G_J(A)) \neq 0$ then $BH^i(G_I(B)) = 0 $ for $i \leq r$ and $BH^{r+1}(G_I(B)) \neq 0$.
\end{example}

\begin{remark}\label{vashi2}
The assertion in Remark \ref{zeroTOn} follows from \ref{vashi} and an easy induction on the number of variables attached.
\end{remark}
\section{ Asymptotic Bockstein cohomology }

It is well known that for all $n \gg 0$ we have $\depth G_{I^n}(A)  \geq 1$. However there are examples where
$H^1(G_{I^n}(A))$ has infinite length for all $n \geq 1$. Thus  it is possible to have $\depth G_{I^n}(A)  = 1$ for all $n \gg 0$.

For Bockstein Cohomolgy we have the following result:
\begin{theorem}\label{asymp}
Let $(A,\m)$ be a local ring  and let $M$ be a \CM \ $A$-module of dimension $d \geq 2$. Let $I$ be an ideal of definition for $M$.
Then $BH^1(G_{I^n}(M)) = 0$ for all $n \gg 0$.
\end{theorem}
\begin{proof}
We may assume that residue field of $A$ is infinite and that $I$ is $\m$-primary.
Choose $n \gg 0$ such that 
$$H^0(L^{I^n}(M)) = 0 \quad \text{and} \quad  H^1(L^{I^n}(M))_j = 0 \ \text{for all} \ j \geq 0; \ \text{see \ref{Veronese}}.$$
We may also assume that $H^2(G_{I^n}(A))_j = 0$ for all $j \geq 1$, see \cite[18.3.13]{BSh}.
Set $K = I^n$. Let $x$ be a $M$-superficial with respect to $K$.  Let $N = M/xM$. Set  
$$G = G_K(M), \  L = L^K(M), \ \ov{G} = G_K(N) \ \text{and} \ \ov{L} = L^K(N).$$
By \ref{dag} we have an exact sequence
\[
0 \rt H^1(G)_j \rt H^1(L)_j.
\]
So $H^1(G)_j = 0$ for $j \geq 0$ and $\ell(H^0(G)_{-1}) \leq \ell( H^0(L)_{-1})$.
By \ref{longH} we have an exact sequence
\[
0 \rt H^0(\ov{L})_n \rt H^1(L)_{n-1} \rt H^1(L)_n.
\]
So $H^0(\ov{L})_n = 0$ for $n > 0$ and
\[
H^1(L)_{-1} = H^0(\ov{L})_0  = \frac{\widetilde{KN}}{KN}.
\]
We also have an exact sequence 
\[
0 \rt H^0(\ov{G})_n \rt H^1(G)_{n-1} \xrightarrow{x^*} H^1(G)_n  \rt H^1(\ov{G})_n \rt H^2(G)_{n-1} \xrightarrow{x^*} H^2(G)_n.
\]
so we obtain 
\[
H^1(G)_{-1} \cong H^0(\ov{G})_0  = \frac{\widetilde{KN}}{KN} \quad \text{and}  \  H^1(\ov{G})_{1} \cong H^2(G)_0.
\]
Furthermore the map
\begin{equation*}
  H^1(G)_{n-1} \xrightarrow{x^*} H^1(G)_n \ \ \  \text{ is injective for} \ \  n < 0. \tag{$\dagger$} 
\end{equation*}
We now consider the defining exact sequence for Bockstein cohomology of $\ov{G}$. Set $\ov{L_1} = L_1^K(N)$.
Note we have an exact sequence
\[
o \rt \ov{G} \rt \ov{L_1} \rt  \ov {G}(-1) \rt 0
\]
Since $H^0(\ov{L_1})_1 \subseteq H^0(\ov{L})_1 = 0$ we have an inclusion
\[
\ov{\beta}^0_1 \colon H^0(G(-1))_1 = H^0(G)_0 \rt H^1(\ov{G})_1.
\]
By \ref{mod} we have a commutative diagram
\[
\xymatrix
{
& H^{0}(\ov{G})_0
\ar@{->}[r]
\ar@{->}[d]^{\ov{\beta}^0_1}
& H^{1}(G)_{-1}
\ar@{->}[d]^{\beta^1_0}
\\
& H^{1}(\ov{G})_1
\ar@{->}[r]
& H^{2}(G)_0
}
\]
Note the horizontal maps are  isomorphisms. Since $\ov{\beta}^0_1$ is an inclusion we have that $\beta^1_0$ is an inclusion.

We now  prove that $\beta^1 \colon H^1(G)(-1) \rt  H^2(G)$ is injective.  It suffices to prove that the graded components of $\beta^1$, i.e., the
maps $\beta^1_n \colon H^1(G)_{n-1} \rt H^2(G)_n$ is injective for all $n \in \Z$. As $H^1(G)_n = 0$ for $n \geq  0$ we have trivially that $\ker \beta^1_n = 0$ for $n \geq 1$.
  By induction on $m  = -n$ we prove that $\ker \beta^1_{-m} =  0$ for all $m \geq 0$.   By the above argument we have that $\ker \beta^1_0 = 0$.
  Assume that $m > 0$ and that $\beta^1_{-m+1}$ is injective.  By \ref{mod} we have a commutative diagram
  \[
\xymatrix
{
& H^{1}(G)_{-m-1}
\ar@{->}[r]^{x^*}
\ar@{->}[d]^{\beta^1_{-m}}
& H^{1}(G)_{-m}
\ar@{->}[d]^{\beta^1_{-m+1}}
\\
& H^{2}(G)_{-m}
\ar@{->}[r]^{x^*}
& H^{2}(G)_{-m+1}
}
\]
By ($\dagger$)   the top row is injective. As $\beta^1_{-m+1}$ is injective we have that $\beta^1_{-m}$ is an inclusion. Thus by induction we have shown that $\ker \beta^1_{-m} = 0$ for all
$m \geq 0$. It follows that $\ker \beta^1 = 0$. Thus $BH^1(G_K(M)) = 0$.
 \end{proof}

\textit{Conjecture} (with assumptions as in Theorem)
$$\ell\left(BH^1(G_I(M))\right) < \infty. $$

\section{$e_2^I(M) = 0$ or $I$ is an integrally closed ideal \\ with $e_2^I(A) = e_1^I(A)- e_0^I(A) + \ell(A/I).$}
In this section we prove the following theorem:
\begin{theorem}\label{e2}
Let $(A,\m)$ be a local ring and let $M$ be a $3$-dimensional \CM \ $A$-module. Let $I$ be an ideal of definition for $M$. Assume that any one of the following conditions hold
\begin{enumerate}[\rm (1)]
\item
$e_2^I(M) = 0$.
\item
$M = A$, the ideal $I$ is an integrally closed   with $e_2^I(A) = e_1^I(A)- e_0^I(A) + \ell(A/I).$
\end{enumerate}
Then for all $n \gg 0$ we have $BH^i(G_{I^n}(M)) = 0$ for $i <3$.
\end{theorem}

We now state a more general result which implies Theorem \ref{e2}.
\begin{theorem}\label{gen} 
Let $(A,\m)$ be a local ring and  let $M$ be a \CM \ module of dimension $d \geq 3$.  Let $I$ be an ideal of definition for $M$ and let $x$ be $M$-superficial \wrt \ $I$. Set $N = M/xM$. If
 $H^1(L^I(N))  = 0$  then
 for all $n \gg 0$ we have $BH^i(G_{I^n}(M)) = 0$ for $i <3$.
\end{theorem} 
\begin{proof} 
Since $H^1(L^I(N)) = 0$ it follows from \ref{longH} and \ref{Artin-vanish} that $H^2(L^I(M)) = 0$. We also have an exact sequence
\[
0 \rt H^0(L^I(N))_n \rt H^1(L^I(M))_{n-1} \rt H^1(L^I(M))_n \rt 0
\]
It follows that $H^1(L^I(M))_n  \cong H^1(L^I(M))_{-1}$ for all $n < 0$.

We now chose $n \gg 0$ such that
\[
H^0(L^{I^n}(M)) = 0 \quad \text{and} \quad H^1(L^{I^n}(M))_j = 0 \ \text{for} \   j \geq 0.
\]
We may also assume that $BH^1(G_{I^n}(M)) = 0$.
Set  $K = I^n$. Also note that
\[
H^2(L^K(M)) = 0 \quad \text{and} \quad H^1(L^K(M))_j  \cong H^1(L^K(M))_{-1} \ \text{for} \  j < 0. 
\]
Note $x^n$ is $M$-superficial \wrt \ $K$. Set $E = M/x^n M$.   Also set $L = L^K(M)$ and $\ov{L}  = L^K(E)$. We have an exact sequence
\[
0 \rt H^0(\ov{L})_j \rt H^1(L)_{j-1} \rt H^1(L)_j \rt H^1(\ov{L})_j \rt 0.
\]
So we have $H^1(L)_{-1} \cong H^0(\ov{L})_0 =  \widetilde{KE} /KE$. Also note that as $H^1(L)_{-2} \cong H^1(L)_{-1}$ we get that
$H^1(\ov{L})_{-1} = 0$. It follows that $\xi_K(E) \geq 2$; see \ref{xi}.  It follows that $H^1(\ov{L})_j  = 0$ for $j < 0$. 

Set $G = G_K(M)$ and $\ov{G} = G_K(E)$. As $\xi_K(E) \geq 2$ we have that $H^1(\ov{G})_j = 0$ for $j < 0$ see \ref{xi}. Also note that we have an exact sequence
\[
0 \rt H^1(G)_j \rt H^1(L)_j
\]
So we have $H^1(G)_j = 0$ for $j  \geq 0$.
We  also have an exact sequence
\[
0 \rt H^0(\ov{G})_j \rt H^1(G)_{j-1} \rt H^1(G)_j \rt H^1(\ov{G})_j
\]
It follows that $H^1(G)_{-1}  \cong H^0(\ov{G})_0   = \widetilde{KE}/KE$ and $H^1(G)_j \cong H^1(G)_{-1}$ for $j < 0$.
Thus $H^1(G)_j \cong H^1(L)_j $ for all $j < 0$. By the exact sequence
\[
 0 \rt H^1(G)_j \rt H^1(L)_j \rt H^1(L)_{j-1} \rt H^2(G)_j \rt 0,
\]
 for $j \leq 0$ we have $H^2(G)_j \cong  H^1(L)_{j-1} \cong H^1(G)_{j-1}$. Also note that $H^2(G)_j = 0$ for $j> 0$.
As $H^2(L^K) = 0$, we get  $\beta^2 = 0$; see \ref{vanishing}. Also as $BH^1(G) = 0$ and $H^0(G) = 0$ we have that
$\beta^1$ is injective. 
Consider $\beta^1 \colon H^1(G)(-1) \rt H^2(G)$. Let $\beta^1_j \colon H^1(G)_{j-1} \rt H^2(G)_j$ be a component of $\beta^1$.  Note $\beta^1_j$ is injective as $\beta^1$ is injective. Also as $\ell(H^2(G)_j) = \ell(H^1(G)_{j-1})$ for all $j$ we get that $\beta^1_j$ is an isomorphism. Thus $\beta^1$ is an isomorphism. It follows that
$BH^2(G) =0$.
\end{proof}
We now prove Theorem \ref{e2}.
\begin{proof}[Proof of Theorem \ref{e2}]
By 1.3 we may assume that residue field of $A$ is infinite. So there exists  $x \in I$ an $M$-superficial element  \wrt \ $I$. By \cite[Lemma 11]{It} we may assume that if $I$ is integrally closed then the $A/(x)$ ideal $I/(x)$ is also integrally closed. By \cite[4.5]{Pu6}  the conditions of Theorem \ref{gen} are satisfied. The result follows. 
\end{proof}
\section{Normal ideals}
Recall that an ideal $I$ is said to be normal(asymptotically normal)  if $I^n$ is integrally closed for all $n$ ( for all $n \gg 0$).  Huneke and Huckaba showed that if $\dim A \geq 3$ is \CM \ and $I$ is an $\m$-primary normal ideal then $\depth G_{I^n}(A) \geq 2$ for all $n \gg 0$. They also gave an example of a normal ideal in a $3$-dimensional  \CM \ ring with
$\depth G_{I^n}(A) = 2$ for all $n\geq 1$.  In this section our result is
\begin{theorem}\label{normal}
Let $(A,\m)$ be a \CM \ local ring of dimension $d \geq 3$ and let $I$ be an $\m$-primary normal ideal. Then
for all $n \gg 0$ the Bockstein cohomology modules $BH^i(G_{I^n}(A)) = 0$ for $i = 0, 1,2$. 
\end{theorem}
\begin{proof}
We may assume that the residue field is infinite. 
Assume $n \gg 0$ such that 
$$\depth G_{I^n}(A)  \geq 2  \quad \text{and} \   H^2(L^{I^n})_j = 0 \ \text{for} \ j \geq 0. $$
We also assume,  see \cite[18.3.13]{BSh},  that 
$$H^3(G_{I^n}(A))_j  = 0 \  \text{for} \ j \geq 1. $$
Set $K = I^n$. Let $x,y$ be a $K$-superficial sequence. Set $B =  A/(x)$ and $C = A/(x,y)$. Note we can assume, perhaps going to a faithfully flat extension, that $J = K/(x)$ is an asymptotically normal ideal, see \cite[Theorem 1]{ItN}.  So $\xi_K(B) \geq 2 $.  In particular we have 
\[
H^1(L^K(B))_j   =  H^1(G_K(B))_j = 0  \ \text{for} \  j < 0;  \text{see  \ref{xi}}. 
\]
Set 
$$L = L^K(A), G = G_K(A),  \ov{L} = L^K(B), \ov{G} = G_K(B), L^* = L^K(C) \ \text{and} \ G^* = G_K(C) . $$
We have an exact sequence
\[
0 \rt H^1(\ov{L})_j \rt  H^2(L)_{j-1} \rt H^2(L)_j.
\]
It follows that $H^1(\ov{L})_j = 0$ for $j \geq 1$. We also have an exact sequence
\[
0 \rt H^0(L^*)_j \rt H^1(\ov{L})_{j-1} \rt H^1(\ov{L})_j.
\]
It follows that $H^0(L^*)_j = 0$ for $j = 0$ and $j \geq 2$. Furthermore
$$H^1(\ov{L})_0 \cong H^0(L^*)_1 =  \frac{\widetilde{K^2C}}{K^2C}. $$
By using the exact sequence
\[
0 \rt H^1(\ov{G})_j \rt H^1(\ov{L})_j \rt H^1(\ov{L})_{j-1},
\]  
we get that $H^1(\ov{G})_j = 0$ for $j \geq 1$ and
\[
H^1(\ov{G})_0 \cong H^1(\ov{L})_0 \cong \frac{\widetilde{K^2C}}{K^2C}.
\]
We have an exact sequence
\[
0 \rt H^2(G)_n \rt H^2(L)_n.
\]
So $H^2(G)_n = 0$ for $n \geq 0$.
We now consider the exact sequence
\[
0 \rt H^1(\ov{G})_n \rt H^2(G)_{n-1} \xrightarrow{x^*} H^2(G)_n \rt H^2(\ov{G})_n \rt H^3(G)_{n-1} \xrightarrow{x^*} H^3(G)_n \rt 0.
\]
It follows that
\[
H^1(\ov{G})_0 \cong H^2(G)_{-1} \quad \text{and} \quad H^2(\ov{G})_1 \cong H^3(G)_0.
\]
As $H^1(\ov{G})_n = 0$ for $n < 0$ we have the map 
\begin{equation*}
H^2(G)_{n-1} \xrightarrow{x^*} H^2(G)_n \quad \text{is injective for} \ n < 0. \tag{$\dagger$}
\end{equation*}
Set $\ov{L_1} = L^K_1(B)$. We consider the defining exact sequence for Bockstein cohomology of $\ov{G}$,
\[
0 \rt \ov{G} \rt H^1(\ov{L_1}) \rt \ov{G}(-1) \rt 0.
\]
Since $H^1(\ov{L_1})_1 \subseteq H^1(\ov{L})_1  = 0$ we get that the map
\[
\ov{\beta}^1_1 \colon H^1(\ov{G})_0 \rt H^2(\ov{G})_1 \quad \text{is injective}.
\]
By \ref{mod} we have a commutative diagram
\[
\xymatrix
{
& H^{1}(\ov{G})_0
\ar@{->}[r]
\ar@{->}[d]^{\ov{\beta}^1_1}
& H^{2}(G)_{-1}
\ar@{->}[d]^{\beta^2_0}
\\
& H^{2}(\ov{G})_1
\ar@{->}[r]
& H^{3}(G)_0
}
\]
Note the horizontal maps are  isomorphisms. Since $\ov{\beta}^1_1$ is an inclusion we have that $\beta^2_0$ is an inclusion.

We now  prove that $\beta^2 \colon H^2(G)(-1) \rt  H^3(G)$ is injective.  It suffices to prove that the graded components of $\beta^2$, i.e., the
maps $\beta^2_n \colon H^2(G)_{n-1} \rt H^3(G)_n$ is injective for all $n \in \Z$. As $H^2(G)_n = 0$ for $n \geq 0$ we have trivially that $\ker \beta^2_n = 0$ for $n \geq 1$.
  By induction on $m  = -n$ we prove that $\ker \beta^2_{-m} =  0$ for all $m \geq 0$.   By the above argument we have that $\ker \beta^2_0 = 0$.
  Assume that $m > 0$ and that $\beta^2_{-m+1}$ is injective.  By \ref{mod} we have a commutative diagram
  \[
\xymatrix
{
& H^{2}(G)_{-m-1}
\ar@{->}[r]^{x^*}
\ar@{->}[d]^{\beta^2_{-m}}
& H^{2}(G)_{-m}
\ar@{->}[d]^{\beta^2_{-m+1}}
\\
& H^{3}(G)_{-m}
\ar@{->}[r]^{x^*}
& H^{3}(G)_{-m+1}
}
\]
By $(\dagger)$ the top row is injective. As $\beta^2_{-m+1}$ is injective we have that $\beta^2_{-m}$ is an inclusion. Thus by induction we have shown that $\ker \beta^2_{-m} = 0$ for all
$m \geq 0$. It follows that $\ker \beta^2 = 0$. Thus $BH^2(G_K(A)) = 0$.
\end{proof}

\section{ $e_2 = e_1 - e_0 + 1, e_2 =3$ and dim $A = 3$}
In this section we prove the following result
\begin{theorem}\label{chi2}
Let $(A,\m)$ be Cohen-Macaulay of dimension three. If $e_2 = e_1 - e_0 + 1$ and $ e_2 = 3$ then
$BH^i(G_\m(A)) = 0$ for $i \leq 2$.
\end{theorem}
To prove this result we need the following:
\begin{proposition}\label{dim2}
Let $(A,\m)$ be a \CM  \ local ring of dimension two and with an infinite residue field. Assume $e_2 = e_1 - e_0 + 1 = 3$.  Set $G = G_\m(A)$. Then
\begin{enumerate}[\rm(1)]
\item
$\depth G = 0$ or $2$.
\item
If $\depth G = 0$ then $\widetilde{\m^2} \neq \m^2$ and $\widetilde{\m^j} = \m^j$ for all $j \neq 2$. Furthermore $\ell(\widetilde{\m^2}/\m^2) = 1$.
\end{enumerate}
\end{proposition}
\begin{proof}
Let $x,y$ be an $A$-superficial sequence  \wrt \ $\m$. Set $J = (x,y)$, $B = A/(x)$ and $\n  = \m/(x)$. 

(1) If $\depth G = 1$ then we have $e_2(B) = e_1(B) - e_0(B) + 1$. Using \cite[Proposition 13]{Pu1} it follows that $\n^3 = y\n^2$. So by \cite[2.1]{Sally2} we have that $\depth G_\n(B) = 1$. So by Sally descent we have $\depth G = 2$ a contradiction.

(2) For $j \geq 0$ set $\sigma_j = \ell(\widetilde{\m^{j+1}}/J\widetilde{\m^{j}})$. Then by \cite[Theorem 3]{It} we have $e_1 = \sum_{j \geq 0} \sigma_j$ and $e_2 = \sum_{j \geq 1}j\sigma_j$. Since $e_2 = e_1 - e_0 + 1$ we get $\sigma_j = 0$ for $j \geq 2$ and
$e_2 = \sigma_1 = \ell(\widetilde{\m^{2}}/J\m) = 3$. 

If $\widetilde{\m^2} = \m^2$ then as $\sigma_2 = 0$ we get $\widetilde{\m^3} = J\m^2 $. It follows that $\widetilde{\m^3} = \m^{3}$. As $\sigma_j = 0$ for $j \geq 2$ inductively one can show that $\widetilde{\m^j} = \m^j$ for $j \geq 2$. It follows that $\depth G \geq 1$, a contradiction. Therefore $\widetilde{\m^2} \neq \m^2$.

Note 
\[
3 = e_2 = \ell(\widetilde{\m^{2}}/J\m) = \ell(\widetilde{\m^2}/\m^2) + \ell(\m^2/J\m).
\]
It follows that $\ell(\m^2/J\m) \leq 2$. If $\ell(\m^2/J\m) \leq 1$ then by  \cite{rv1} or \cite{W} we have $\depth G \geq 1$, a contradiction. Thus $\ell(\m^2/J\m) = 2$. It follows that $\ell(\widetilde{\m^2}/\m^2) = 1$.
 
 As $e_i(B) = e_i$ for $i \leq 1$, \cite[Corollary 10]{Pu1} we get $e_1(B) = e_0(B) + 2$. 
 Note $e_1(B) = \sum_{j\geq 0}\ell(\n^{j+1}/y\n^j)$. Also $\ell(\n^2/y\n) = \ell(\m^2/J\m) = 2$. It follows that $\ell(\n^3/y\n^2) = 1$ and $\n^4 = y \n^3$. 
 Thus $e_2(B) =  \sum_{j\geq 1}j\ell(\n^{j+1}/y\n^j) = 4$.

By  \cite[Corollary 10]{Pu1} we have
\[
e_2 = e_2(B) - \sum_{n \geq 2} \ell((\m^{n+1} \colon x)/\m^n).
\]
So $ \sum_{n\geq 2} \ell((\m^{n+1} \colon x)/\m^n) = 1$.  We have an exact sequence, see \cite[p.\ 305]{rv2}
\[
0 \rt \frac{(\m^3 \colon x)}{\m^2} \rt \frac{\m^3}{J\m^2} \rt \frac{\n^3}{y\n^2} \rt 0.
\]
If $(\m^3 \colon x) = \m^2$ then $\ell(\m^3/J\m^2) = 1$. So by a result due to Huckaba \cite{H} we get $\depth G \geq 1$, a contradiction. So $(\m^3 \colon x) \neq \m^2$.
It follows that $(\m^{n+1} \colon x) = \m^n $ for $n \geq 3$.  
For all $i \geq 0$ we have an exact sequence, see \cite[2.6]{Pu5},
 \begin{equation}\label{rr-x2}
  0 \rt \frac{(\m^{i+1}\colon x)}{\m^{i}} \rt \frac{\widetilde{\m^i}}{\m^i}  \rt \frac{\widetilde{\m^{i+1}}}{\m^{i+1}} 
 \end{equation}
 It follows that $\widetilde{\m^j} = \m^j$ for $j \geq 3$.
\end{proof}
We now give
\begin{proof}[Proof of Theorem \ref{chi2}]
Set $G = G_\m(A)$.
By Proposition 8.2 it easily follows that  $\depth G = 0,1$ or $3$.

\textit{Case 1}: $\depth G = 3$ \\
Then $H^i(G) = 0$ for $i<3$. It follows that $BH^i(G) = 0$ for $i < 3$.

 \textit{Case 2}:  $\depth G = 0$. \\
 We may assume that the residue field of $A$ is infinite, see 1.3.  Let $x\in \m$ be $\m$-superficial. Set $B= A/(x)$ and $\n = \m/(x)$.  Also set $\ov{G} = G_\n(B)$, $L = L^\m(A)$ and 
 $\ov{L} = L^\n(B)$.
 By Sally descent $\depth \ov{G} =0$.  By \cite[4.4]{Pu6}
 we have $H^1(\ov{L}) = 0$. So $H^2(L) = 0$. By \ref{dim2} we also have that $\widetilde{\n^j} = \n^j$ for $j \neq 2$ and
 $\widetilde{\n^2}/\n^2 \cong k$. 
 For all $i \geq 0$ we have an exact sequence, see \cite[2.9]{Pu5},
 \begin{equation}\label{rr-x}
  0 \rt \frac{(\m^{i+1}\colon x)}{\m^{i}} \rt \frac{\widetilde{\m^i}}{\m^i}  \rt \frac{\widetilde{\m^{i+1}}}{\m^{i+1}} \rt \frac{\widetilde{\n^{i+1}}}{\n^{i+1}}
 \end{equation}
 Claim 1: $\widetilde{\m^2} \neq \m^2$.
 If this is not the case then as $\widetilde{\n^j} = \n^j$ for all $j \geq 3$ we have that $\widetilde{\m^j} =\m^j$ for all $j \geq 3$.  Also trivially $\widetilde{\m} = \m$. Thus 
 $\widetilde{\m^j} = \m^j$ for all $j \geq 1$. It follows that $\depth G \geq 1$. This is a contradiction. Thus $\widetilde{\m^2} \neq \m^2$.
 
 Since $\ell(\widetilde{\n^2} /\n^2) = 1$, $\widetilde{\m} = \m$ and $\widetilde{\m^2}  \neq \m^2$  we have an isomorphism 
 \[
 \frac{\widetilde{\m^2}}{\m^2}  \rt \frac{\widetilde{\n^2}}{\n^2}
 \]
 It follows that
 \[
 \ov{\widetilde{\m^2}} =  \widetilde{\n^2}.
 \]
 Also as $\widetilde{\n^j} = \n^j$ for $j \geq 3$ we have that 
 \[
 \ov{\widetilde{\m^j}} =  \widetilde{\n^j}   \quad \text{for all} \ j.
 \]
 Thus the natural map $H^0(L) \rt H^0(\ov{L})$ is surjective. By \ref{longH} we have an inclusion $H^1(L)(-1) \rt H^1(L)$. By \ref{Artin-vanish} it follows that
   $H^1(L) = 0$.  As $H^2(L) = 0$ also, it follows from \ref{dag} that $H^2(G) = 0$. So trivially we have
  that $BH^2(G) = 0$. As $H^1(L) = 0$ we have that $\beta^1 = 0$.
  
  By (\ref{rr-x}) we also have that for $j \geq 2$
  $$\widetilde{\m^{j+1}}  = x \widetilde{\m^{j}}  + \m^{j+1}. $$
  As $\widetilde{\m^2} \subseteq \m$, iteratively we have that
  $$ \widetilde{\m^{j+1}} \subseteq \m^j \quad \text{for all} \ j \geq 2. $$ 
  It follows that $BH^0(G) = 0$, i.e., $\beta^0 $ is injective. We also have an exact sequence
  \[
  0 \rt H^0(G) \rt H^0(L) \rt H^0(L)(-1) \rt H^1(G) \rt H^1(L) =0.
  \]
  Notice $H^0(G) = H^0(L)$. It follows that $H^1(G) \cong H^0(G)(-1)$. It follows that $\beta^0$ is surjective too. Thus $BH^1(G) = 0$.
 
\textit{Case 3}: $\depth G = 1$.  \\
Let $x$ be $A$-superficial \wrt \ $\m$.  Set $B= A/(x)$ and $\n = \m/(x)$.  Also set $\ov{G} = G_\n(B)$, $L = L^\m(A)$ and 
 $\ov{L} = L^\n(B)$. Note $x^*$ is $G$-regular.  By 3.5 we have $BH^i(\ov{G}) = 0$ for $i = 0,1$. We will use Sally descent to conclude that $BH^i(G) = 0$ for $i = 0,1,2$.

Set $u = xt \in \R(I)_1$. Since $H^1(\ov{L}) = 0$ we have an exact sequence
\[
0 \rt H^0(\ov{L})_n \rt H^1(L)_{n-1} \xrightarrow{u} H^1(L)_n \rt 0.
\]
Since $H^0(\ov{L})_n = 0$ for $n \geq 2$ we have $H^1(L)_n = 0$ for $n \geq 1$.  We also get
\[
H^1(L)_0 \cong H^0(\ov{L})_1 = \frac{\widetilde{\n^2}}{\n^2}.
\]
Since $H^0(\ov{L})_j = 0$ for $j \leq 0$, we obtain 
isomorphisms
\[
H^1(L)_{n-1} \xrightarrow{u}  H^1(L)_n \quad \text{for} \ n\leq 0
\] 
Thus $$H^1(L)_n  \cong \frac{\widetilde{\n^2}}{\n^2 }  \quad \text{for} \ n\leq 0,$$
and $u$ is $H^1(L)^\vee$-regular. 

We have an exact sequence
\[
0 \rt H^1(G)_n \rt H^1(L)_n.
\]
It follows that $H^1(G)_n = 0$ for $n \geq 1$.
We also have an exact sequence
\[
0 \rt H^0(\ov{G})_n \rt H^1(G)_{n-1} \xrightarrow{x^*} H^1(G)_n \rt H^1(\ov{G})_n.
\] 
So we obtain 
\[
H^1(G)_0 \cong H^0(\ov{G})_1 = \frac{\widetilde{\n^2}}{\n^2}.
\]
Also note that $H^1(\ov{G}) = H^0(\ov{G})(-1)$. It follows that $H^1(\ov{G})_n = 0$ for all $n \neq 2$.  We also have $H^0(\ov{G})_n = 0$ for $n \leq 0$. 
Thus  we obtain isomorphisms
\[
H^1(G)_{n-1} \xrightarrow{x^*}  H^1(G)_n \quad \text{for} \ n\leq 0
\] 
So $$H^1(G)_n  \cong \frac{\widetilde{\n^2}}{\n^2 }  \quad \text{for} \ n\leq 0,$$
and $x^*$ is $H^1(G)^\vee$-regular.

As $H^2(L) = 0$ we obtain an exact sequence
\[
0 \rt H^1(G) \rt H^1(L) \rt H^1(L)(-1) \rt H^2(G) \rt 0.
\] 
Note 
$$H^1(G)_n = H^1(L)_n = 0 \ \text{for} \ n \geq 1 \ \text{and} \ H^1(G)_{n} \cong H^1(L)_n =   \frac{\widetilde{\n^2}}{\n^2} \ \text{for} \ n \leq 0.$$
Thus $H^1(G) \cong H^1(L)$. It follows that $H^2(G) \cong H^1(L)(-1) \cong H^1(G)(-1)$. It follows that $x^*$ is $H^2(G)^\vee$-regular.
Thus by Sally descent we get that $BH^i(G) = 0$ for $i<3$.
\end{proof}

\end{document}